\documentclass[11pt]{amsart}

\usepackage{amsmath,leftidx,amsthm}
\usepackage[psamsfonts]{amssymb}
\usepackage[latin1]{inputenc}
\usepackage{graphicx,color}
\usepackage[all]{xy}

\usepackage{hyperref}

\usepackage[marginpar=2.5cm]{geometry}

\theoremstyle{remark}

\newtheorem*{remark}{\bf Remark}

\newtheorem{pbm}{\bf Problem}

\theoremstyle{plain}

\newtheorem{theorem}{\bf Theorem}
\newtheorem{proposition}[theorem]{\bf Proposition}

\newtheorem{lemma}[theorem]{\bf Lemma}
\newtheorem{corollary}[theorem]{\bf Corollary}

\def\R{{\mathbb R}}

\def\D{\mathbb{D}}

\def\N{{\mathbb N}}

\def\cE{\mathcal{E}}

\def\cX{\mathcal{X}}
\def\cC{\mathcal{C}}

\def\om{\omega}

\DeclareMathOperator{\supp}{supp}

\DeclareMathOperator{\PSH}{PSH}

\def\and{{\quad\text{and}\quad}}

\begin{document}

\title{Regularity of push-forward of Monge-Amp\`ere measures}

\author{Eleonora Di Nezza}
\address{IHES Universit\'e Paris-Saclay, 35 route de Chartres, 91400 Bures sur Yvette, France}
\email{dinezza@ihes.fr}

\author{Charles Favre}
\address{CMLS, \'Ecole polytechnique, CNRS, Universit\'e Paris-Saclay, 91128 Palaiseau Cedex, France}
\email{charles.favre@polytechnique.edu}

\begin{abstract}
We prove that the image under any dominant meromorphic map of the Monge-Amp\`ere measure of a H\"older continuous quasi-psh function still
possesses a H\"older potential. We also discuss the case of lower regularity.
\end{abstract}

\maketitle

\begin{center}
\emph{Dedicated to Jean-Pierre Demailly for his 60th birthday
}
\end{center}

\medskip


\section*{Introduction}

Let $(X,\om_X)$ be a compact K\"ahler manifold of dimension $n$ normalized by  the volume condition  $\int_X \om_X^n =1$. We say that a potential $\varphi \in L^1(X)$ is $\omega_X$-plurisubharmonic ($\omega_X$-psh for short) if locally $\varphi$ is the sum of a plurisubharmonic and a smooth function, and $\omega_X+dd^c \varphi \geq 0$ in the weak sense of currents, where $d= \partial + \bar{\partial}$ and $d^c= \frac{i}{2\pi} ( \bar{\partial}- \partial )$ so that $dd^c = \frac{i}{\pi} \partial \bar{\partial}$. We denote by $\PSH(X, \omega_X)$ the set of all $\omega_X$-psh functions on $X$. Recall from \cite[Section 1]{GZ07} that the non-pluripolar Monge-Amp\`ere measure of a function $\varphi \in \PSH(X, \omega_X)$ is a positive measure defined as the increasing limit
$$
( \om_X + dd^c \varphi)^ n = \lim_{j \to + \infty} \mathbf{1}_{\{\varphi > -j\}} \, \left( \om_X + dd^c \max \{ \varphi, -j\}\right)^ n 
$$
where the right hand side is defined using Bedford-Taylor intersection theory of bounded psh functions, see~\cite{bedford-taylor}. 

One of the main result of \cite{GZ07} states that if $\mu$ is a probability  measure on $X$ which does not charge pluripolar sets, then
there exists a unique (up to  a constant) $\omega_X$-psh function $\varphi$ such that  $\int_X  ( \om_X + dd^c \varphi)^ n= 1$ and
\begin{equation}\label{MA}
\mu = ( \om_X + dd^c \varphi)^ n~.
\end{equation}
We denote by $\mathcal{E}{(X, \omega_X)}$ the set of all $\omega_X$-psh functions whose non-pluripolar Monge-Amp\`ere measure has full mass so that 
any solution to~\eqref{MA} belongs to $\mathcal{E}{(X, \omega_X)}$.

 In the same paper, Guedj and Zeriahi  introduced for any $p>0$ the subset $\mathcal{E}^p(X, \omega_X)$ of $\mathcal{E}(X,\omega_X)$ 
consisting of all $\omega_X$-psh functions satisfying the integrability condition $\varphi \in L^p ((\omega_X+dd^c \varphi)^n)$. 
Since $\omega_X$-psh functions are bounded from above it follows that  
$$  \mathcal{E}^p{(X, \omega_X)}  \subset \mathcal{E}^q{(X, \omega_X)}, \, \text{ for all } p>q.$$
Observe also that  any $\omega_X$-psh function lying in $L^\infty$  belongs to the intersection of all  $ \mathcal{E}^p{(X, \omega_X)} $. 

We shall say that a probability measure which does not charge pluripolar sets $\mu = ( \om_X + dd^c \varphi)^ n$ is a Monge-Amp\`ere measure having  H\"older, continuous, $L^\infty$ or $\cE^p$ potential for some $p>0$ whenever $\varphi$ is  H\"older, continuous, $L^\infty$ or belongs to the energy class $\cE^p(X,\omega_X)$ respectively. 

Let us now consider any dominant meromorphic map $f : X \dashrightarrow Y$ where $(Y,\om_Y)$ is also a compact K\"ahler manifold of volume $1$, and denote by $m$ its complex dimension.  Let $\Gamma$ be a resolution of singularities of the graph of $f$.  We obtain two surjective holomorphic maps $\pi_1 : \Gamma \to X$ and $\pi_2: \Gamma \to Y$
where $\pi_1$ is bimeromorphic so that $\Gamma$ is a modification of a compact K\"ahler manifold. By Hironaka's Chow lemma, see e.g.~\cite[Theorem 2.8]{Peternell} we may suppose that $\pi_1$ is a composition of blow-ups along smooth centers so that $\Gamma$ is itself a compact K\"ahler manifold of complex dimension $n$.  We fix any K\"ahler form $\omega_\Gamma$ on it.

One defines the push-forward under $f$ of a measure $\mu$ not charging pluripolar sets as follows. Since $\pi_1$ is a bimeromorphism, there exist two closed analytic subsets $R \subset \Gamma$ and $V \subset X$ such that $\pi_1: \Gamma \setminus R \to X\setminus V$ is a biholomorphism. One may thus set 
$\pi_1^* \mu$ to be the trivial extension through $R$ of $(\pi_1)|_{\Gamma \setminus R}^* \mu$. This measure is again a probability measure which does not charge pluripolar sets. 

We then define the probability measure $f_* \mu := (\pi_2)_* \pi_1^* \mu$. We observe that since $f$ is dominant then $\pi_2$ is surjective hence the preimage of a pluripolar set in $Y$ by $\pi_2$ is again pluripolar. By the preceding discussion,  there exists $\psi\in \mathcal{E}(Y, \omega_Y)$ such that $f_* \mu =  ( \om_Y + dd^c \psi)^{m}$. 

Our main goal is to discuss the following question. 

\begin{pbm}\label{pbm1}
Suppose $\mu$ is a Monge-Amp\`ere measure having H\"older, continuous, $L^\infty$ or $\cE^p$ potential.
Is it true that $f_* \mu$ is also  a Monge-Amp\`ere measure of a potential lying in the same class of regularity?
\end{pbm}

This problem is hard for Monge-Amp\`ere measures having either continuous or $L^\infty$ potentials since there is no known intrinsic
characterization of these measures.
For these classes of regularity even the case $ f$ is the identity map and $X =Y$ is still open (see for example \cite[Question 15]{open pbms}).
\begin{pbm}\label{pbm2}
Suppose $\mu$ is a probability measure on $X$ not charging pluripolar sets and write $\mu = ( \om + dd^c \varphi)^{ n} =  ( \om' + dd^c \varphi')^{ n}$ where
$\om, \om'$ are two K\"ahler forms of volume $1$. 
Is it true that $\varphi$ is continuous (resp. $L^\infty$) iff  $\varphi'$ is? 
\end{pbm}

An intrinsic characterization of Monge-Amp\`ere measure of H\"older functions is given by~\cite{demailly et al}, and of functions in the energy class $\mathcal{E}^p$ in~\cite[Theorem C]{GZ07}
so that Problem~\ref{pbm2} has a positive answer for these classes of regularity, see~\cite[Theorem 4.1]{DiN}. 
Problem~\ref{pbm1} remains though quite subtle. If we restrict our attention to the regularity in the $\mathcal{E}^p$ energy classes, then the answer is no in general.  
Suppose that $\pi: X \rightarrow \mathbb{P}^2$
is the blow-up at some point $p\in \mathbb{P}^2 $, and let $E=\pi^{-1}(p)$. It was observed by the first author in \cite[Proposition B]{DiN15} that there exists a probability measure $\mu = (\omega_X + dd^c \varphi)^2$ with $\varphi\in \mathcal{E}^1(X, \omega_X) $ but $\pi_* \mu = (\omega_{FS}+ dd^c \psi)^2$ with $\psi\notin \mathcal{E}^1(\mathbb{P}^2, \omega_{FS})$, where $\omega_{FS}$ denotes the Fubini Study metric on $\mathbb{P}^2$ and $\omega_X$ is a K\"ahler form.

\medskip

In this note we answer Problem~\ref{pbm1} in two situations. We first treat the case $\mu$ is the Monge-Amp\`ere of a H\"older function.

\begin{theorem}\label{thm1}
Let $f : X \dashrightarrow Y$ be any dominant meromorphic map between two compact K\"ahler manifolds. If $\mu$ is a Monge-Amp\`ere measure  having a H\"older potential with H\"older exponent $\alpha$, then $f_*\mu$ is a Monge-Amp\`ere measure having a H\"older potential with H\"older exponent bounded by  $C \alpha^{\dim(X)}$ for some constant $C>0$ depending only on $f$.
\end{theorem}

We expect that the technics developed in the paper of Ko{\l}odziej-Nguyen~\cite{kolo-nguyen} in the present volume 
allows one to extend the previous result to arbitrary compact hermitian manifolds. 

Next we treat the case the image of the map has dimension $1$.

\begin{theorem}\label{thm2}
Let $f : X  \dashrightarrow   Y$ be any dominant meromorphic map from a compact K\"ahler manifold to a compact Riemann surface. 
If $\mu$ is a Monge-Amp\`ere measure  having a H\"older, $\cC^0$, $L^\infty$, $\cE^p$ potential respectively, then  $f_* \mu$ is
a Monge-Amp\`ere measure having a potential lying in the same regularity class.
\end{theorem}

Motivations for studying this question come from the analysis of degenerating measures on families of projective manifolds developed in~\cite{favre:hybrid}.
Let us briefly recall the setting of that paper.
Let $ \mathcal{X} $ be a smooth connected complex manifold of dimension $n+1$, and 
$\pi\colon \mathcal{X} \to \D$ be a flat proper analytic map over the unit disk
which is a submersion over the punctured disk and has connected fibers. 
We assume that $\mathcal{X}$ is K\"ahler so that 
 each fiber $X_t = \pi^{-1}(t)$ is also K\"ahler. 

\smallskip

A tame family of Monge-Amp\`ere measures is by definition a family of 
positive measures $\{\mu_t\}_{t\in \D}$ each supported on $X_t$ that can be written under the form
$$\mu_t = p_* (T|_{X_t}^{n})~,$$
where $T$ is a positive closed $(1,1)$-current having local H\"older continuous potentials
and defined on a complex manifold $\cX'$ which admits  a proper bimeromorphic morphism $p\colon\cX'\to \cX$ which is an isomorphism over $X:= \pi^{-1}(\D^*)$. 
It follows from ~\cite[Corollary 1.6]{demailly93} that the family of measures $\mu'_t:= T|_{X_t}^{n}$ in $\cX'$ is continuous so that $\mu'_t $
converges to a positive measure $\mu'_0$
supported on $X_0'$ as $t\to 0$. It follows that the convergence $\lim_{t\to0} \mu_t = \mu_0$ also holds in $\cX$.

As a corollary of the previous results we show the limiting measure $\mu_0$ is of a very special kind:
\begin{corollary}\label{thm:degeneration model}
Let $\{\mu_t\}_{t\in \D}$ be any tame family of Monge-Amp\`ere measures, so that $\mu_t \to \mu_0$ as $t\rightarrow 0.$

Then there exist a finite collection of closed subvarieties $\{V_i\}_{i=0, \ldots, N}$ of $X_0$ and for each index $i$ a positive measure
$\nu_i$ supported on $V_i$ such that
$$
\mu_0 = \sum_{i=1}^N \nu_i
$$
and $\nu_i$ is a Monge-Amp\`ere measure on $V_i$ having a H\"older potential.
\end{corollary}
In the previous statement, it may happen that $V_i$ is singular, in which case it is understood that the pull-back of $\nu_i$ to a (K\"ahler) resolution of $V_i$
is a Monge-Amp\`ere measure having a H\"older continuous potential.

\bigskip

\noindent {\bf Acknowledgement}. We thank Ahmed Zeriahi for useful discussions on these problems.


\section{Images of Monge-Amp\`ere measures having a H\"older potential: proof of Theorem~\ref{thm1}}
 As already mentioned, a dominant meromorphic map $f \colon X \dashrightarrow Y$ can be decomposed as $f= \pi_2 \circ \pi_1^{-1}$, where $\pi_1 \colon \Gamma \to X$ is holomorphic and bimeromorphic and $\pi_2\colon \Gamma \to Y$ is a surjective holomorphic map. Recall that one can assume $\Gamma$ to be K\"ahler, and that $f_* \mu := (\pi_2)_* \pi_1^* \mu$. 
 
We first claim that if $\mu$ is the Monge-Amp\`ere of a H\"older continuous function then $\pi_1^* \mu$ too. Let $\varphi\in \PSH(X, \omega_X)$ be the H\"older potential such that $\mu=(\omega_X+dd^c \varphi)^n$. It then follows from Bedford and Taylor theory that  
$\pi_1^* \mu= (\pi_1^* \omega_X+dd^c \pi_1^* \varphi)^n$. Since $\pi_1^* \omega_X$ is a semipositive smooth form, there exists a positive constant $C>0$ such that $\pi_1^* \mu \leq  (C \omega_{\Gamma}+ dd^c \pi^* \varphi)^n$
where $\omega_\Gamma$ is a K\"ahler form on $\Gamma$, and \cite[Theorem 4.3]{demailly et al} implies that $\tilde{\mu}:=\pi_1^* \mu$ is the Monge-Amp\`ere measure of a H\"older continuous $C\omega_\Gamma$-psh function. This proves the claim. We are then left to prove that $(\pi_2)_* \tilde{\mu}$ is the Monge-Amp\`ere measure of a H\"older potential. This will be done in Lemma \ref{lem:crucial}.

\medskip

We first show that the push-forward of a smooth volume form has density in $L^{1+\varepsilon}$, for some constant $\varepsilon>0$ depending only on $f$. 
\begin{proposition}\label{pro:push volume}
Let $f\colon X \to Y$ be a surjective holomorphic map. Then $f_* \omega_X^n = g \omega_Y^m$ with $g \in L^{1+ \varepsilon}(\omega_Y^{n})$, for some $\varepsilon>0$.
\end{proposition}
This result is basically \cite[Proposition 3.2]{Song-Tian} (see also~\cite[Section 2]{Tosatti}). We give nevertheless a detailed proof for reader's convenience. Pick any coherent ideal sheaf $\mathcal{I}\subset \mathcal{O}_X$, and denote by $V(\mathcal{I})= \supp (\mathcal{O}_X/ \mathcal{I})$ the closed analytic subvariety of $X$ associated to $\mathcal{I}$. Let $\{U_i\}_{i=1}^N$ be a finite open covering of $X$ by balls and $\{V_i\}_i $ be a subcovering such that $\overline{V}_i \subset U_i$. The analytic sheaf $\mathcal{I}$ is globally generated on each $U_i$ so that we can find holomorphic functions such that $\mathcal{I}|_{U_i}= \left(h^{(i)}_1, \dots, h^{(i)}_k \right) \cdot \mathcal{O}_{U_i}$. Let $\{\rho_i\}$ be a partition of unity subordinate to $\overline{V_i}$. We then define
\begin{equation} \label{ideal}
\Phi_\mathcal{I}:= \sum_{i=1}^N \rho_i \left(\sum_{j=1}^k |h^{(i)}_j|^2 \right).
\end{equation} 
Then $\Phi_\mathcal{I} \colon X \to \R_+$ is a smooth function which vanishes exactly on $V(\mathcal{I})$. 
Observe that if $\Phi_\mathcal{I}$ and $\Phi_\mathcal{I}'$ are defined using two different coverings, then there exists $C>0$ such that 
$$
\frac{1}{C} \Phi_\mathcal{I}' \leq \Phi_\mathcal{I} \leq C \Phi_\mathcal{I}'.$$ 
In the sequel we shall abuse notation and not write the dependence of $\Phi_{\mathcal{I}}$ in terms of the local generators of the ideal sheaf. 
The logarithm of the obtained function is then well-defined up to a  bounded function so that all statements in the next Lemma make sense.

\begin{lemma}\label{integrability1} Let $\mathcal{I}, \mathcal{J}\subseteq \mathcal{O}_X$ be two coherent ideal sheafs. The followings hold:
\begin{itemize}
\item[(i)] there exists $\varepsilon>0$ such that $|\Phi_\mathcal{I} |^{-\varepsilon}\in L^1 (X)$;
\item[(ii)] if $\mathcal{I} \subseteq \mathcal{J}$ then $\Phi_\mathcal{J} \geq c \Phi_\mathcal{I}$ for some positive $c>0$;
\item[(iii)] if $V(\mathcal{J}) \subseteq V(\mathcal{I}) $ then there exists $c, \theta>0$ such that $\Phi_\mathcal{J} \geq c \,\Phi_\mathcal{I}^\theta$;
\item[(iv)] given $f\colon X\rightarrow Y$ a holomorphic surjective map and a coherent ideal sheaf $ \mathcal{J}\subseteq \mathcal{O}_Y$, then $\Phi_{f^*\mathcal{J}} = \Phi_{\mathcal{J}} \circ f   $ (for a suitable choice of local generators of $\mathcal{J}$ and $f^*\mathcal{J}$).
\end{itemize}
\end{lemma}
\begin{proof}
Using a resolution of singularities of $\mathcal{I}$, one sees that the statement in $(i)$ reduces to show that $|z_1|^{-\varepsilon}$ is locally integrable for some $\varepsilon>0$, and this is the case if we choose $\varepsilon$ small enough.
The statements in $(ii)$ and $ (iv) $ follow straightforward from the definition in \eqref{ideal}. The statement in $(iii)$ is a consequence of {\L}ojasiewicz theorem, see e.g. \cite[Theorem 7.2]{LJTR}.
\end{proof}

\begin{lemma}\label{integrability}
Let $f\colon X \rightarrow Y$ be a holomorphic surjective map and let $\mathcal{I}\subseteq \mathcal{O}_X$ be a coherent ideal sheaf. Then there exists a coherent ideal sheaf $\mathcal{J}\subseteq \mathcal{O}_Y$, and constants $c,\theta>0$ such that for any $y\in Y$ we have
$$\inf_{x\in f^{-1}(y)} \Phi_{\mathcal{I}} \geq c\, \Phi_{\mathcal{J}}^{\theta} $$
\end{lemma}
\begin{proof}
Let $\mathcal{J}\subseteq \mathcal{O}_Y$ be the coherent ideal sheaf of holomorphic functions vanishing on the  set  $f(V(\mathcal{I}))$ which is analytic since $f$ is proper.
Observe that  $V(f^*\mathcal{J})= f^{-1}(V(\mathcal{J}))\supset V(\mathcal{I})$, so that Lemma \ref{integrability1} ((iii) and (iv)) insure that there exist $c, \theta>0$ such that
$$ \Phi_{\mathcal{I}} \geq c \Phi_{f^*\mathcal{J}}^\theta = c (\Phi_{\mathcal{J}} \circ f )^\theta.$$
Hence the conclusion.
\end{proof}

\begin{proof}[Proof of Proposition \ref{pro:push volume}]
Recall that Sard's theorem implies the existence of a closed subvariety $S\subsetneq Y$ such that $f\colon X\setminus f^{-1}(S) \rightarrow Y\setminus S$ is a submersion. 

We first prove that $f_*  \omega_X^n $ is absolutely continuous w.r.t. $\omega_Y^m$. We need to check that $\omega_Y^m(E)=0$ implies $f_*  \omega_X^n  (E)=0$ for any Borel subset $E\subset Y$. 
As $S$ and $f^{-1}(S)$ have volume zero one may assume that $f$ is a submersion in which case the claim follows from Fubini's theorem.
 
Radon-Nikodym theorem  now guarantees that $f_*  \omega_X^n  = g\omega_Y^m $ for some $0\leq g\in L^1(Y)$.  We want to show that the integral
$$\int_Y g^{1+\varepsilon} \omega_Y^m=\int_Y g^\varepsilon f_*  \omega_X^n=\int_X( f^* g)^\varepsilon  \,  \omega_X^n$$
is finite for some $\varepsilon>0$ small enough. 
Consider the smooth function  $\phi(x):= \frac{f^*\omega_Y^m \wedge \omega_X^{n-m} } { \omega_X^n} (x)$, and set $ \tilde{\phi} (y) := \inf_{x\in f^{-1}(y)} \phi(x)$ so that $\phi \geq f^* \tilde{\phi}$. We claim that for any $y\in Y$ 
\begin{eqnarray}\label{bound density}
 g(y) \leq  \frac{1}{\tilde{\phi} (y)}.
\end{eqnarray}
Let $\chi$ be a test function (i.e. a non negative smooth function) on $Y$, then 
\begin{eqnarray*}
\int_Y \chi \,g \omega_Y^n & = & \int_X f^*\chi \,\omega_X^n
= \int_X \frac{f^*\chi}{\phi} \,f^*\omega_Y^m \wedge \omega_X^{n-m}  \\
&\leq & \int_X  f^* \left(\frac{\chi}{\tilde{\phi}} \omega_Y^m \right) \wedge \omega_X^{n-m}\\
&\mathop{=}\limits^{\text{Fubini}} & C(f) \int_Y \frac{\chi}{\tilde{\phi}} \omega_Y^m
\leq  \int_Y \frac{\chi}{\tilde{\phi}} \omega_Y^m
\end{eqnarray*}
where $C(f)$ is the volume of a fiber over a generic point $y\in Y$, i.e. $C(f)= \int_{f^{-1}(y)}  \omega_X^{n-m}$. The last inequality follows from the fact $C(f)\leq 1$ since $\int_X \omega_X^{n}=1$. The claim is thus proved. Lemma \ref{integrability1} (i) and (iv) combined with  Lemma \ref{integrability} then insure that there exists $\varepsilon>0$ such that $(f^*g)^\varepsilon\in L^1(\omega_X^n)$.
\end{proof}
Theorem~\ref{thm1}  is reduced to the following result which relies in an essential way on Proposition~\ref{pro:push volume}.

\begin{proposition}\label{lem:crucial}
Suppose $ f\colon X \to Y$ is a surjective holomorphic map between compact K\"ahler manifolds. 
If $\mu$ is a positive measure on $X$ with H\"older continuous potentials, then $f_*\mu$
is a positive measure on $Y$ with H\"older potentials.
\end{proposition}

Observe that by multiplying $\omega_X$ by a suitable positive constant  we may assume that $ f^* \om_{Y} \le \om_{X}$. The volume normalization is no longer satisfied
but  a positive multiple of $\mu$ is still the Monge-Amp\`ere measure of a $\omega_X$-psh H\"older continuous function. 
Write $f_*\mu = (\om_{Y} + dd^c \psi)^{m}$ with $\psi\in  \PSH(Y, \om_{Y})$.

We claim that there exists $C>0$, and  $\varepsilon>0$  such that for all $u \in  \PSH(Y, \om_{Y})$ with $\int_X u \,\omega_X^n= 0$
\begin{equation}\label{eq:uniform-int}
\int_{Y} \exp ( - \varepsilon u ) \, d(f_* \mu) \le C~.
\end{equation}
Indeed, for any $u \in \PSH(Y, \om_{Y})$ we have that 
$$
\int_{Y}e^{-\varepsilon u} \, d(f_* \mu) = \int_{X} e^{-\varepsilon (u \circ f)} \, d\mu~.
$$
Now the integral $\int_{X}e^{-\varepsilon( u \circ f)} \, d\mu$ is uniformly bounded by~\cite[Theorem~1.1]{DNS10} since:
\begin{itemize}
\item $\mu$ has H\"older continuous potentials; 
\item $f^* \om_{Y} \le \om_{X}$ hence
$u \circ f \in \PSH(X,\om_{X})$;
\item and the set of functions in $\PSH(X, \om_{X})$ such that $\int_X u \,\omega_X^n= 0$ is compact by~\cite[Proposition~2.6]{GZ05}.
\end{itemize}
This proves our claim. Using the terminology of \cite{DN14} this means that $f_* \mu $ is \emph{moderate}.
It is worth mentioning that if~\cite[Question 16]{open pbms} holds true then the conclusion of Proposition \ref{lem:crucial} would follow immediately since any moderate measure would have a H\"older continuous potential. To get around this problem we use the characterization of measures with H\"older potentials given by Dinh and Nguyen.

\begin{proof}[Proof of Lemma~\ref{lem:crucial}]

By \cite[Lemma 3.3]{DN14},  $f_* \mu$ is the Monge-Amp\` ere measure of a H\"older potential if and only if there exist $\tilde{c}>1$ and $\tilde{\beta}\in (0,1)$ such that
\begin{equation}\label{eq the goal}
\int_Y |u-v| \,df_*\mu \leq \tilde{c} \max \left(\|u-v\|_{L^1(\omega_Y^n)}, \|u-v\|^{\tilde{\beta}}_{L^1(\omega_Y^n)}\right)
\end{equation}
for all $u, v \in \PSH (Y, \omega_Y) $. By assumption on $\mu$ we know there exist $c>1$ and $\beta\in (0,1)$ such that
$\int_Y |u-v| \,df_*\mu = \int_X  |f^*u-f^*v| \,d\mu$, and
\begin{equation}\label{eq:dinh}
\int_X  |f^*u-f^*v| \,d\mu \leq c \max \left(\|f^*u-f^*v\|_{L^1(\omega_X^n)}, \|f^*u-f^*v\|^\beta_{L^1(\omega_X^n)}\right). 
\end{equation}
Also, Proposition \ref{pro:push volume} gives
\begin{equation}\label{eq 1}
 \int_X|f^*u-f^*v| \omega_X^n=  \int_Y |u-v| \, g\, \omega_Y^n \leq   \| g \|_{L^{1+\varepsilon}(\omega_Y^n ) }   \|u-v\|_{L^p(\omega_Y^n)} 
 \end{equation}
where $p$ is the conjugate exponent of $1+\varepsilon$. Set $C_g:= \| g \|_{L^{1+\varepsilon}(\omega_Y^n ) } <+\infty $. Up to replace $C_g$ with $C_g+1$ we can assume that $C_g \geq 1$.

 Denote by $m_u:= \int_Y u \, \omega_Y^n$ and observe that $u':=u-m_u $, $v':=v-m_v$ satisfy $\int_X u'\, \omega_X^n=0 =\int_X v'\, \omega_X^n$. Then the triangle inequality gives
\begin{eqnarray}\label{eq 2}
 \nonumber \|u-v\|_{L^p(\omega_Y^n)}  & = & \left(\int_Y |(u'-v') + (m_u-m_v)|^p \omega_Y^n \right)^{1/p}\\
 \nonumber & \leq &  \|u'-v'\|_{L^p(\omega_Y^n)} + |m_u-m_v|\\
  & \leq &  \|u'-v'\|_{L^p(\omega_Y^n)} + \|u-v\|_{L^1(\omega_Y^n)}.
\end{eqnarray}
At this point, we make use of \cite[Proposition 3.2]{DN14} (that holds for normalized potentials) to replace the $L^p$-norm with the $L^1$-norm. We then infer the existence of a constant $c'>1$ such that
\begin{eqnarray*}
 \nonumber  \|u'-v'\|_{L^p(\omega_Y^n)} & \leq & c' \max(1, -\log  \|u'-v'\|_{L^1(\omega_Y^n)})^{\frac{p-1}{p}}  \|u'-v'\|_{L^1(\omega_Y^n)}^\frac{1}{p}.
\end{eqnarray*}
When $t :=\|u'-v'\|_{L^1(\omega_Y^n)}\geq 1/e$ we clearly have
$$  \|u'-v'\|_{L^p(\omega_Y^n)}\leq  c' \|u'-v'\|_{L^1(\omega_Y^n)}^\frac{1}{p},$$
whereas for any integer $N\in\N^*$, there exists a constant $c_N>0$ such that  $-\log t \leq c_N t^{-1/N}$ when $t\le 1/e$, hence
$$  \|u'-v'\|_{L^p(\omega_Y^n)}\leq  c'' \|u'-v'\|_{L^1(\omega_Y^n)}^{\frac{1}{p}\left( 1- \frac{p-1}{N}  \right)}~.$$
As $ \|u'-v'\|_{L^1(\omega_Y^n)} \leq  2\|u-v\|_{L^1(\omega_Y^n)}$,  combining \eqref{eq:dinh}, \eqref{eq 1} and \eqref{eq 2} we get
\begin{align*}
\|f^*u-f^*v\|_{L^1(\mu)} 
&\leq C \max\left(  \|u-v\|_{L^1(\omega_Y^n)}^{\tilde{\beta}},  \|u-v\|_{L^1(\omega_Y^n)}   \right)
~, 
\end{align*}
with $\tilde{\beta} = \frac{\beta}{p}\left( 1- \frac{p-1}{N}  \right)$.
By \cite[Lemma 3.3]{DN14} $f_* \mu= (\omega_Y+dd^c \psi)^n$ where $\psi$ is a H\"older continuous function.

To get a bound on the H\"older regularity of $\psi$, one argues as follows. 
First if $\mu = (\om + dd^c \varphi)^n$ with $\varphi$ a $\alpha$-H\"older potential, and $\pi\colon  \Gamma \to X$ is a proper modification, then 
$\pi^* \mu$ is dominated by a Monge-Amp\`ere measure with $\alpha$-H\"older potential, and \cite[Proposition 3.3 (ii)]{demailly et al} is satisfied with $b = 2\alpha/(\alpha + 2n)$ by~\cite[Theorem 4.3 (iii)]{demailly et al}.
Hence, following the proof of ~\cite[Theorem]{demailly et al}, we see that $\pi^* \mu$ is a Monge-Amp\`ere measure of a $\alpha_1$-H\"older continuous
potential with $\alpha_1 <b/(n+1) $ (see Remark below for more details about the latter statement).

By~\cite[Proposition 4.1]{DN14},~\eqref{eq:dinh} holds with $\beta = \alpha^n_1/(2 + \alpha_1^n)$, and~\eqref{eq the goal} is then satisfied for any 
$\tilde{\beta}<\beta/p$ so that $f_*\mu$ is a Monge-Amp\`ere measure with $\tilde{\alpha}$-H\"older potential for any $\tilde{\alpha} < 2 \tilde{\beta}/(m+1)$, see the discussion on~\cite[p.83]{DN14}. Combining all these estimates we see that any 
$$
\tilde{\alpha} < \frac{\alpha^n}{p(m+1)(\alpha/2+n)^n(n+1)^{n}}
$$
works where $p$ is the conjugate of the larger constant $\varepsilon>0$ for which Proposition~\ref{pro:push volume} holds.
\end{proof}
\begin{remark}
We borrow notations from the proof of \cite[Theorem A]{demailly et al}. Fix $\alpha_1< b/(n+1)$ and choose $\varepsilon>0$ such that $\alpha_1\leq \alpha\leq \alpha_0\leq b-\alpha_0(n+\varepsilon)$. By the previous arguments we know that condition $(ii)$ in \cite[Proposition 3.3]{demailly et al} holds, i.e. for any $\phi\in \PSH(\Gamma, \omega_\Gamma)$, we have $\|\rho_\delta \phi- \phi\|_{L^1(\pi^* \mu)}= O(\delta^b)$, where $b =2\alpha/(\alpha + 2n)$. In particular, this gives 
$$\pi^* \mu (E_0) \leq c_1 \delta^{b-\alpha_0}.$$ 
Let $g\in L^1(\pi^*\mu)$ be defined as $g=0$ on $E_0$ and $g= c$ on $\Gamma\setminus E_0$ where $c$ is a positive constant such that $\pi^*\mu(\Gamma)= \int_\Gamma g\, d( \pi^* \mu)$. An easy computation gives that $c= \pi^*\mu(\Gamma)/ \pi^*\mu(\Gamma\setminus E_0)$. Let $v\in \PSH(\Gamma,\omega_\Gamma)$ be the bounded solution of the Monge-Amp\`ere equation $(\omega_\Gamma+dd^c v)^n= g \cdot \pi^* \mu$. Observe that $$\|1-g\|_{L^1(\pi^* \mu)} = \int_{E_0} d  \pi^*\mu + \int_{\Gamma \setminus E_0} |1-c|\, d  \pi^*\mu =2 \int_{E_0} d\pi^*\mu \leq 2 c_1 \delta^{b-\alpha_0}.$$
Since $\pi^* \mu= (\omega_\Gamma+dd^c \tilde{\varphi})^n$ satisfies the $\mathcal{H}(\infty)$ property we can still apply \cite[Theorem 1.1]{DZ} and get $$\|\tilde{\varphi}-v\|_{L^\infty} \leq c_3 \delta^{\frac{b-\alpha_0}{n+\varepsilon}}.$$
The exact same arguments as in \cite[Theorem A]{demailly et al} then insure that the H\"older exponent of $\tilde{\varphi}$ is $\alpha_1$.
\end{remark}


\section{Over a one-dimensional base: proof of Theorem~\ref{thm2}}

In this section we treat Problem~\ref{pbm1} in the case the base is a Riemann surface. 

We start with the case of a surjective holomorphic map $ f\colon  X \to Y$ from a K\"ahler compact manifold to a compact Riemann surface. 

\medskip
%
%
%

Let $\mu = (\om_X + dd^c \varphi)^{ n}$ be  a Monge-Amp\`ere measure of a continuous $\om_X$-psh function $\varphi$.
Suppose $v_k, v$ is a family of $\om_X$-psh functions such that $v_k \to v$ in $L^1$, then 
\begin{align*}
\int_X v_k \, d\mu 
&= 
\int_X v_k \, (\omega_X + dd^c \varphi)^{ n}
\\
&=\int_X v_k \, \omega_X^{n}
+ \sum_{j=0}^{n-1} \int_X \varphi\, dd^c v_k \wedge \omega_X^j \wedge (\omega_X + dd^c \varphi)^{ n-j-1}
\to  \int_X v \, d\mu 
\end{align*}
by ~\cite[Corollary~1.6 (a)]{demailly93}. Observe that in the last equality we used the fact that $$(\omega_X + dd^c \varphi)^{ n} - \omega_X^n = \sum_{j=0}^{n-1} dd^c \varphi \wedge \omega_X^j \wedge (\omega_X + dd^c \varphi)^{ n-j-1}$$ and Stokes' theorem.

Normalize the K\"ahler form on $Y$ such that $\int \om_Y =1$,  
and pick any sequence $y_k\to y_\infty\in Y$.  Let $w_k$ be the 
solutions of the equations $\Delta w_k = \delta_{y_k} - \om_Y$ with $\sup w_k =0$ so that 
$w_k (y) - \log |y-y_k|$ is continuous in local coordinates near $y_k$. 
Write $f_*\mu = \om_Y + dd^c \psi$ so that 
$$\int_Y w_k \, d(f_*\mu) =\int_Y w_k \, \om_Y + \int_Y \psi \, \Delta w_k
 = \psi(y_k) + \int_Y (w_k-\psi)  \om_Y ~.$$
Since $w_k\to w_\infty$ in $L^p_{\mathrm{loc}}$ for all $p<\infty$, Proposition~\ref{pro:push volume} 
implies that $ f^*w_k \to  f^*w_\infty$ in the $L^1$ topology, 
so that the argument above gives $\int_Y w_k \,d(f_*\mu)=  \int_X  f^* w_k \, d\mu \to \int_X  f^* w_\infty \, d\mu=\int_Y w_\infty\, d(f_*\mu)$ 
We then conclude that $\psi(y_k) \to \psi(y_\infty)$. Hence $\psi$ is continuous.

\medskip

Suppose then that $\mu$ is locally  the Monge-Amp\`ere of a bounded psh function, and pick any subharmonic function $u$ defined in a neighborhood of a point $y\in Y$. 
Then $f^* u$ is again psh in a neighborhood of $f^{-1}(y)$, and the standard Chern-Levine-Nirenberg inequalities imply that $f^*u \in L^1(\mu)$ so that $u \in L^1(f_*\mu)$
with  a norm depending only on the $L^1$-norm of $u$.
It follows that $f_*\mu$ is locally the laplacian of a bounded subharmonic function.

\medskip

Finally, assume $\mu=(\omega_X+ dd^c \varphi)^n$ for some $\varphi\in \mathcal{E}^p(X, \omega_X)$. By \cite[Theorem C]{GZ07} this is equivalent to have that $\mathcal{E}^p(X,{\omega_X}) \subset L^p (\mu)$. Write as usual  $f_* \mu= (\omega_Y+dd^c \psi)$ with $\psi\in  \mathcal{E}(Y, \omega_Y)$.

We claim that $u\in \mathcal{E}^p(Y, \omega_Y)$ implies $f^* u\in \mathcal{E}^p(X, \omega_X)$. Indeed, without loss of generality we can assume that $\Omega:= \omega_X - f^*\omega_Y$ is a K\"ahler form and by the multilinearity of the non-pluripolar product we have
\begin{align*}
\int_X |f^* u|^p (\omega_X+dd^c f^* u)^n &= \int_X  |f^* u|^p (f^*\omega_Y + \Omega + dd^c f^* u)^n\\
&= \int_X  |f^* u|^p \left(\Omega^n + (f^*\omega_Y + dd^c f^* u) \wedge \Omega ^{n-1} \right)
\end{align*}
where the last identity follows from the fact that $(f^*\omega_Y + dd^c f^* u)^j=0$ for $j>1$. The term $\int_X  |f^* u|^p \Omega ^{n}$ is bounded thanks to the integrability properties of quasi-plurisubharmonic functions w.r.t. volume forms \cite[Theorem 1.47]{GZ17}; while the term
$$\int_X  |f^* u|^p (f^*\omega_Y + dd^c f^* u)\wedge \Omega ^{n-1} = C(f) \int_Y |u|^p \,(\omega_Y + dd^c u)$$ is finite since $u\in \mathcal{E}^p(Y, \omega_Y)$. This proves the claim. 

Now, given any $u\in \mathcal{E}^p(Y, \omega_Y)$ we have
$$\int_Y |u|^p\,  d(f_* \mu)= \int_X |f^* u|^p d\mu<+\infty$$
since $f^* u\in \mathcal{E}^p(X, \omega_X) \subset L^p(\mu)$. The conclusion follows from \cite[Theorem C]{GZ07}.

\vspace{5mm}

Consider now any dominant meromorphic map  $ f\colon  X  \dashrightarrow  Y$ from a K\"ahler compact manifold to a compact Riemann surface. As above we decompose $f$ such that $f_* \mu =(\pi_2)_* \pi_1^* \mu$ for any positive measure $\mu$ on $X$.

Assume that $\mu$ has continuous potentials. If we write $\mu=(\omega_X+dd^c \varphi)^n $ then
$\pi_1^*\mu= (\pi_1^* \omega_X+dd^c \varphi\circ \pi)^n \leq (C\omega_\Gamma+ dd^c \varphi\circ \pi)^n:= \hat{\mu}$ where $ \hat{\mu}$ has a continuous potential. This implies $f_* \mu \leq (\pi_2)_* \hat{\mu}$. Observe that by the previous arguments  the measure $(\pi_2)_* \hat{\mu}$ has continuous potential. It follows that locally $f_* \mu =\Delta v \leq \Delta u$ where $u,v$ are subharmonic functions.   It follows that $v$ is the sum of a continuous function and the opposite of a subharmonic (hence u.s.c.) function. 
Since it is also u.s.c we conclude to its continuity. 

When $\mu$ has bounded potentials, the same argument applies noting that subharmonic functions are always bounded from above which implies $v$ to be bounded. 

Finally, we consider the case where $\mu$ is the Monge-Amp\`ere measure of $\varphi\in \mathcal{E}^p(X, \omega_X)$. We first observe that given $v\in  \mathcal{E}^p(\Gamma, \omega_\Gamma)$ we have $(\pi_1)_* v \in \mathcal{E}^p(X, \omega_X)$. Indeed,
$$\int_X |v \circ \pi^{-1}|^p (\omega_X +dd^c v\circ \pi^{-1})^n = \int_\Gamma |v|^p (\pi_1^* \omega_X +dd^c v)^n \leq \int_\Gamma |v|^p ( C\omega_\Gamma +dd^c v)^n < +\infty.$$
This and the previous arguments give that if $u\in \mathcal{E}^p(Y, \omega_Y)$ then $f_* u=(\pi_1)_* \pi_2^* u \in \mathcal{E}^p(X, \omega_X)$, hence
$$\int_Y |u|^p d f_* \mu = \int_X |u\circ f|^p d\mu <+\infty.  $$
It follows from \cite[Theorem C]{GZ07} that $f_* \mu $ is the Monge-Amp\`ere measure of a function in $\mathcal{E}^p(Y, \omega_Y)$.

\section{The case of submersions}
In this section we let $(X, \omega_X), (Y, \omega_Y)$ be two compact K\"ahler manifolds of dimension $n$ and $m$, respectively and normalized such that $\int_X \omega_X^n=1= \int_Y \omega_Y^m$.

\begin{proposition}
Let $f\colon  X \to Y$ be a submersion. Then, $u\in \cE^p(Y, \omega_Y)$ implies $f^* u  \in \cE^p (X, \omega_X)$. In particular, if  a probability measure $\mu$ is the Monge-Amp\`ere of a function in $\mathcal{E}^p$ then also $f_* \mu $ has also a potential  in $\mathcal{E}^p.$
\end{proposition}
\begin{proof}
Since $f$ is a submersion we can assume that there is a finite number of open neighbourhoods $U_i$ such that $X\subset \bigcup_{j=0}^N U_j$, $f|_{U_j}(z,w)=z$ where $z=(z_1, \dots, z_m)$ and $w=(z_{m+1}, \dots, z_n)$. Moreover  we can assume that on each $U_j$ we have 
$$\omega_X\leq C_j \frac{i}{2}\left(dz \wedge d\bar{z} + dw \wedge d\bar{w}\right), \qquad \frac{i}{2}dz \wedge d\bar{z}\leq A_j f^* \omega_Y$$ where $A_j,C_j>1$ and $dz\wedge d\bar{z}, dw \wedge d\bar{w}$ are short notations for $\sum_{j=1}^m dz_j\wedge d\bar{z_j}$ and $\sum_{k=m+1}^n dz_k\wedge d\bar{z_k}$, respectively. We then write
\begin{align*}
\int_X |f^* u|^p (\omega_X+dd^c f^* u)^n &\leq  \sum_{j=1}^N  \int_{U_j} |f^* u|^p \left(C_j  \frac{i}{2}dz \wedge d\bar{z} + C_j \frac{i}{2}dw \wedge d\bar{w}+dd^c f^* u\right)^n\\
&\leq  \sum_{j=1}^N  \int_{U_j} |f^* u|^p \left(  A'_j f^* \omega_Y +C_j \frac{i}{2}dw \wedge d\bar{w}+dd^c f^* u\right)^n\\
&= \sum_{j=1}^N  \sum_{\ell=0}^n\int_{U_j} |f^* u|^p \left(  A'_j f^* \omega_Y +dd^c f^* u\right)^\ell \wedge \left(C_j \frac{i}{2}dw \wedge d\bar{w}\right)^{n-\ell}\\
&= \sum_{j=1}^N \int_{U_j} |f^* u|^p \left(  A'_j f^* \omega_Y +dd^c f^* u\right)^m\wedge \left(C_j \frac{i}{2}dw \wedge d\bar{w}\right)^{n-m}.
\end{align*}
The above integral is then finite because by assumption $u\in \mathcal{E}^p(Y, A\omega_Y)$ for any $A\geq 1$.

The last statement follows from the same arguments in the last part of the proof in the previous section.
\end{proof}


\section{Tame families of Monge-Amp\`ere measures: proof of Corollary~\ref{thm:degeneration model}}

Recall the setting from the introduction: $ \mathcal{X} $ is a smooth connected complex manifold of dimension $n+1$, and 
$\pi\colon  \mathcal{X} \to \D$ is a flat proper analytic map over the unit disk
which is a submersion over the punctured disk and has connected fibers. 
We let $p\colon  \cX' \to \cX$ be a proper bi-meromorphic map from a smooth complex manifold $\cX'$
which is an isomorphism over $\pi^{-1}(\D^*)$.

We let $T$ be any closed positive $(1,1)$-current on $\cX'$ admitting local H\"older continuous potentials.
Observe that by e.g.~\cite[Corollary 1.6]{demailly93} we have  $$\mu'_t =  dd^c \log|\pi\circ p -t| \wedge T^{n} \to \mu'_0 := dd^c \log|\pi\circ p| \wedge T^{n}~.$$

Let us now analyze the structure of the positive measure 
$\mu_0:= p_*\mu'_0$. First observe that $\mu'_0$ can be decomposed as a finite sum of positive measures
$\mu'_E :=  (T|_{E})^{n}$ where the sum is taken over all irreducible components $E$ of $\cX'_0$. 
Each of these measures is locally the Monge-Amp\`ere of a H\"older continuous psh function. 

Write $V := p(E)$. Since $E$ is irreducible, $V$ is also an irreducible (possibly singular) subvariety of dimension $\ell$. To conclude the proof it remains to show that 
$p_*(\mu'_E)$ is the Monge-Amp\`ere measure of H\"older continuous function that is locally the sum of a smooth and psh function. More precisely, one needs to show that $p_*(\mu'_E)$ does not charge any proper algebraic
subset of $V$, and given any resolution of singularities $\varpi \colon  V' \to V$ the pull-back measure $\varpi^* (p_*(\mu'_E))$ can be locally written as $(dd^c u)^\ell$ where $u$ is a H\"older psh function on $V'$.

This follows from Theorem~\ref{thm1}  applied to any resolution of singularities $V'$ of $V$ and to any $E'$ which admits a birational morphism
$E'\to E$ such that the map $E' \to V'$ induced by $p$ is also a morphism.



\end{document}